\documentclass[12pt]{amsart}

\usepackage{amsmath, amsthm, amssymb, amsfonts, amscd, graphicx, endnotes, tikz, color, hyperref}

\usetikzlibrary{tikzmark,calc}

\usepackage{mathtools}

\usepackage[margin=01.1in]{geometry}

\usetikzlibrary{arrows}

\usepackage[nodisplayskipstretch]{setspace}
\def\AA{{\mathbb A}}

\def\CC{{\mathbb C}}
\def\FF{{\mathbb F}}

\def\KK{{\mathbb K}}

\def\QQ{{\mathbb Q}}

\def\QQ{{\mathbb Q}}
\def\RR{{\mathbb R}}
\def\TT{{\mathbb T}}
\def\ZZ{{\mathbb Z}}

\def\0{{\mathbf 0}}
\def\1{{\mathbf 1}}

\def\Dcal{{\mathcal D}}

\def\Ocal{{\mathcal O}}

\def\Xcal{{\mathcal X}}

\def\Beer{\mathrm{Beer}}
\def\Taib{\mathrm{Taib}}
\def\Berk{\mathrm{Berk}}
\def\Fourier{\mathrm{Fourier}}

\theoremstyle{plain}

\newtheorem{thm}{Theorem}

\newtheorem{prop}[thm]{Proposition}
\newtheorem{lem}[thm]{Lemma}

\theoremstyle{definition}

\newtheorem*{Koksma}{Non-Archimedean Koksma Inequality}

\newtheorem{ex}{Example}

\title[Non-Archimedean Koksma inequalities]{Non-Archimedean Koksma inequalities, \\ variation, and Fourier analysis}

\author{Clayton Petsche}

\address{Clayton Petsche; Department of Mathematics; Oregon State University; Corvallis OR 97331 U.S.A.}

\email{petschec@math.oregonstate.edu}

\author{Naveen Somasunderam}

\address{Naveen Somasunderam; Mathematics Department; SUNY Plattsburgh; Plattsburgh NY 12901; U.S.A.}

\email{nsoma001@plattsburgh.edu}

\date{March 14, 2022}

\begin{document}

\begin{abstract}
We examine four different notions of variation for real-valued functions defined on the compact ring of integers of a non-Archimedean local field, with an emphasis on regularity properties of functions with finite variation, and on establishing non-Archimedean Koksma inequalities. The first version of variation is due to Taibleson, the second due to Beer, and the remaining two are new.  Taibleson variation is the simplest of these, but it is a coarse measure of irregularity and it does not admit a Koksma inequality.  Beer variation can be used to prove a Koksma inequality, but it is order-dependent and not translation invariant.  We define a new version of variation which may be interpreted as the graph-theoretic variation when a function is naturally extended to a certain subtree of the Berkovich affine line.  This variation is order-free and translation invariant, and it admits a Koksma inequality which, for a certain natural family of examples, is always sharper than Beer's.  Finally, we define a Fourier-analytic variation and a corresponding Koksma inequality which is sometimes sharper than the Berkovich-analytic inequality.
\end{abstract}

\maketitle


\section{Introduction}

Let $K$ be a locally compact field equipped with a nontrivial, non-Archimedean absolute value $|\cdot|$.  Examples include the $p$-adic field $\QQ_p$ for a prime number $p$, or more generally any finite extension of $\QQ_p$, as well as the field of fractions $\FF_q(\!(T)\!)$ of the formal power series ring $\FF_q[\![T]\!]$ over a finite field $\FF_q$.  (According to the classification theorem for local fields (\cite{MR1680912} $\S$ 4.2), these examples exhaust all possibilities).  Let $\Ocal=\{x\in K\mid |x|\leq1\}$ be the compact ring of integers in $K$, and let $\mu$ be the Haar measure on $\Ocal$ normalized so that $\mu(\Ocal)=1$.

Let $\{X_n\}$ be a sequence of finite subsets of $\Ocal$ such that $|X_n|\to+\infty$.  Such a sequence is said to be {\bf equidistributed} in $\Ocal$ if, for every disc $D\subseteq\Ocal$, we have
\begin{equation}\label{EquidistributionDefinition1}
\lim_{n\to+\infty}\frac{|X_n\cap D|}{|X_n|}=\mu(D).
\end{equation}
A standard approximation argument can be used to show that an equivalent condition characterizing equidistribution is that, for all continuous functions $f:\Ocal\to\RR$, we have
\begin{equation}\label{EquidistributionDefinition2}
\lim_{n\to+\infty}\frac{1}{|X_n|}\sum_{x\in X_n}f(x)=\int_\Ocal f\,d\mu.
\end{equation}
Thus if a finite set $X$ is distributed nearly uniformly throughout $\Ocal$, then the Riemann-type sum $(1/|X|)\sum_{x\in X}f(x)$ should closely approximate the integral $\int_\Ocal f\,d\mu$.  

To make this idea quantitative, one defines the {\bf discrepancy} of a finite subset $X\subseteq\Ocal$ by 
\begin{equation}\label{DiscrepancyDefinition}
\Delta(X)=\sup_{D\subseteq \Ocal}\left|\frac{|X\cap D|}{|X|}-\mu(D)\right|,
\end{equation}
the supremum over all discs $D\subseteq\Ocal$.  Thus for a given finite set of points $X$, the quantity $\Delta(X)$ measures the maximal difference, over all discs $D\subseteq\Ocal$, between the actual proportion of the points occurring in the disc $D$ and the expected proportion.  It is then desirable to obtain a result of the following type.

\begin{Koksma}
For a particular class of sufficiently regular functions $f:\Ocal\to\RR$, the inequality
\begin{equation}\label{KoksmaGeneral}
\left|\frac{1}{|X|}\sum_{x\in X}f(x)-\int_\Ocal f\,d\mu\right|\leq C(f)\Delta(X)
\end{equation}
holds, where $C(f)$ is a constant depending only on $f$.  
\end{Koksma}

In the well-developed study of equidistribution on the circle group $\RR/\ZZ$, the inequality analogous to (\ref{KoksmaGeneral}) is due to Koksma (see \cite{MR0419394} $\S$ 2.5), and the quantity $C(f)$ is (up to a constant) the real-analytic variation $V(f)=\int_{\RR/\ZZ}|f'(x)|\,dx$ of $f$.  Consequently, the real integral $\int_{\RR/\ZZ}f(x)\,dx$ can be approximated numerically using low-discrepancy sequences, an idea which forms the basis of quasi-Monte Carlo integration.  Niederreiter \cite{MR508447} has provided an exposition of quasi-Monte Carlo methods, and Morokoff-Caflisch \cite{MR1365433} have done an extensive experimental study of the effectiveness of quasi-Monte Carlo methods in evaluating integrals in single and multi-dimensions. 

In the non-Archimedean setting, there are multiple ways one might define a notion of the variation of a function $f:\Ocal\to\RR$.  In this paper we examine four possibilities for such a definition, with an emphasis on regularity properties of functions with finite variation, and also with an eye toward establishing Koksma inequalities of the form (\ref{KoksmaGeneral}).  

The first two notions of variation we consider were initiated by Taibleson \cite{MR217522} and Beer \cite{MR237441}, both in the 1960s.  Despite working at nearly the same time, these two authors seem to have been unaware of each other's work.  One of our goals is to give an overview of their ideas and to provide a comparison of their relative strengths and weaknesses.

Taibleson's version of variation, which we consider in $\S$~\ref{TaibVarSection}, has perhaps the simplest and most elegant definition.  One considers any partition of $\Ocal$ into a finite collection discs, next sums the maximum differences in the values taken by $f$ in each disc, and finally takes the supremum over all such partitions.  This variation is typically the easiest to calculate in specific examples, and it is translation invariant with respect to the group structure on $\Ocal$.  Moreover, functions with finite Taibleson variation have at most countably many discontinuities, as we show in $\S$~\ref{TaibVarSection}, and they satisfy a decay condition on their Fourier coefficients, as was shown by Taibleson \cite{MR217522}.

Taibleson variation is majorized by all three of the other notions of variation we consider, but no general inequalities exist in the opposite direction, which may be an indication that Taibleson variation is a rather coarse measure of irregularity.  Indeed, we will see examples in which Taibleson variation fails to detect a certain type of oscillation which the other notions of variation are sensitive to.  Because of this lack of sensitivity to oscillation, we can show that it is actually impossible to prove a Koksma inequality of the form (\ref{KoksmaGeneral}) in which the constant $C(f)$ depends only on the Taibleson variation of $f$.  Thus for the purposes of establishing a Koksma inequality, a different idea is needed.

In $\S$~\ref{BeerVarSection} we describe the approach of Beer, who was the first to prove a Koksma inequality of the form (\ref{KoksmaGeneral}).  Beer's construction involves selecting a certain dictionary ordering on $\Ocal$ which allows one to (essentially) identify $\Ocal$ with the real unit interval $[0,1]$ using base $q$ expansions of real numbers.  Then one can use a real analytic argument to emulate the proof of the classical Koksma inequality.  One drawback of Beer's approach is that the value of the Beer variation of a function is not absolute, but rather it depends on the arbitrary choice of an ordering on $\Ocal$.  Moreover, Beer variation is not translation invariant, which is unfortunate because $\Ocal$ is a group.  

In $\S$~\ref{BerkVarSection} we define a new notion of variation, which eliminates the drawbacks suffered by both the Beer and Taibleson variations.  This new notion may be interpreted as the graph-theoretic variation when a function $f:\Ocal\to\RR$ is naturally extended to a function on a certain subtree of the Berkovich affine line associated to the local field $K$.  Like Taibleson variation, this Berkovich-analytic variation is order-free and translation invariant.  But we are also able to prove a Koksma inequality for Berkovich-analytic variation, in contrast to Taibleson variation, for which no Koksma inequality is possible.  We also prove a strong regularity condition for functions of finite Berkovich-analytic variation, showing that every such function is equal almost everywhere to a continuous function.

In $\S$~\ref{FourierVarSection} we give a fourth approach, proving a Koksma inequality of the type (\ref{KoksmaGeneral}) using Fourier analysis on the compact group $\Ocal$.  We declare the constant $C(f)$ that arises in this inequality the Fourier-analytic variation of $f$; roughly speaking, this constant is finite when the Fourier coefficients of $f$ decay rapidly enough.  Like Berkovich analytic variation, the Fourier-analytic variation of a function is always larger than the Taibleson variation, and functions with finite Fourier-analytic variation satisfy a strong regularity property.

In $\S$~\ref{ComparingSection}, we summarize how the three known non-Archimedean Koksma inequalities compare with one another for the sample application $f(x)=|x-c|^t$ for $c\in\Ocal$ and $t>0$.  In this family, our Berkovich-analytic Koksma inequality is always sharper than Beer's Koksma inequality.  Our Fourier-analytic Koksma inequality is sharper than Beer's result for large $t$, and sharper than both Beer's result and our Berkovich-analytic result as $t\to0$.  

Niederreiter \cite{MR249370} has derived a general Koksma inequality on compact abelian groups, as well as a Fourier analytic Koksma inequality on $\RR/\ZZ$. Niederreiter's approach does not involve any notion of discrepancy defined directly on the group $G$, but instead considers the distribution of points on the unit circle under the character maps from $G$ to the unit circle. 

We thank David Finch for helpful suggestions about this work.

\section{Review of Local fields}\label{PrelimSection}

Throughout this paper $K$ is a field which is locally compact with respect to a nontrivial, non-Archimedean absolute value $|\cdot|$, and $\Ocal=\{x\in K\mid |x|\leq1\}$ is its compact ring of integers.  Let $\pi\in \Ocal$ be a uniformizing parameter; this means that $|\pi|$ is maximal among all $x\in \Ocal$ with $|x|<1$, and $\pi \Ocal$ is the unique maximal ideal of $\Ocal$.  Because $K$ is assumed to be locally compact, the residue field $\Ocal/\pi \Ocal$ must be finite, and we denote its order by $q$.  We may assume without loss of generality that the absolute value $|\cdot|$ is normalized so that $|\pi|=1/q$.  

Given an element $a\in K$ and a real number $r>0$ with $r\in q^{\ZZ}$, denote by
\begin{equation*}
D_r(a) = \{x\in K\,\mid\,|x-a|\leq r\}
\end{equation*} 
the closed disc in $K$ with center $a$ and radius $r$.  Letting $\mu$ denote the Haar measure on $\Ocal$, normalized so that $\mu(\Ocal)=1$, the normalization of the absolute value implies that the Haar measure of a disc is the same as its radius; that is $\mu(D_r(a))=r$.

It is useful to fix a complete set $S$ of coset representatives in $\Ocal$ for the residue field $\Ocal/\pi \Ocal$, with the assumption that $0\in S$.  As is well-known, each element $x\in \Ocal$ can be written uniquely as a power series $x=\sum_{k\geq0}a_k\pi^k$ in $\pi$, for $a_k\in S$; see \cite{MR1680912} $\S$ 4.2.

If $A$ is a subset of $\Ocal$, we denote by $\Xcal_A(x)$ the characteristic function of $A$.

\section{Taibleson variation}\label{TaibVarSection}

The first notion of variation we consider, due to Taibleson \cite{MR217522}, has the simplest definition.  Taibleson was considered only the case of functions defined on the formal power series ring $\Ocal=\FF_p[\![T]\!]$ over a finite field $\FF_p$, but extending this definition to our more general setting is straightforward.

By a {\bf Taibleson partition} of $\Ocal$ we mean any finite collection $\Pi$ of discs which form a partition of $\Ocal$.  Given a function $f: \Ocal \to \RR$ and a Taibleson partition $\Pi$ of $\Ocal$, define
\begin{equation*}   
\begin{split}
 V_\Pi(f) & = \sum_{D\in \Pi}\sup_{x,y \in D} (f(x) - f(y)).
\end{split}
\end{equation*}
Define the {\bf Taibleson variation} of $f$ by $V_\Taib(f) = \sup_\Pi V_\Pi(f)$, the supremum over all Taibleson partitions of $\Ocal$.

\begin{ex}
Consider the characteristic function $f=\Xcal_{A}:\Ocal\to\RR$ of a proper subdisc $A\subsetneq \Ocal$. We show that $V_\Taib(f) =1$.  By taking the partition to be just $\Pi = \{ \Ocal \}$, we see that $V_\Pi(f) = 1$.  For any other partition $\Pi$, note that either $A$ is strictly contained in some $D\in\Pi$, or some $D\in \Pi$ is contained in $A$.  If $A$ is strictly contained in some $D\in\Pi$, then $V_\Pi(f) = 1$, but if $A$ contains some $D\in\Pi$, then $f$ is constant on each $D\in \Pi$ and so $V_\Pi(f) = 0$.  We conclude that  $V_\Taib(f) = 1$.  
\end{ex}

\begin{ex} 
Let $c\in \Ocal$, let $t>0$, and consider the function $f:\Ocal\to\RR$ defined by $f(x) = |x-c|^t$.  We will show that $V_\Taib(f) =1$.  By taking the partition to be just $\Pi = \{ \Ocal \}$, we see that $V_\Pi(f) = 1$.  For any other partition $\Pi$, let $D_c$ be the disc containing $c$, and suppose that $D_c$ has radius $r$.  Then $f$ is constant on each $D\neq D_0$, and it follows that $V_\Pi(f) \leq r^t\leq 1$.  We conclude that $V_\Taib(f)=1$.
\end{ex}

\begin{prop}\label{TaibContinuousProp}
If a function $f:\Ocal\to\RR$ has finite Taibleson variation, then there exists a countable subset $Z$ of $\Ocal$ such that $f$ is continuous at every point in $\Ocal\setminus Z$.
\end{prop}
\begin{proof}
For each $x\in\Ocal$, define $\phi(x)=\limsup_{y\to x}|f(y)-f(x)|$, and note that $f$ is continuous precisely at those $x\in\Ocal$ with $\phi(x)=0$.  Assume that the set $Z=\{x\in\Ocal\mid \phi(x)>0\}$ is uncountable.  Then there exists $\epsilon>0$ and an infinite sequence $\{x_m\}_{m=1}^\infty$ of distinct points in $Z$ with $\phi(x_m)\geq\epsilon$ for all $m$.  (If no such $\epsilon>0$ existed, then $Z=\cup_{k\geq1}\{x\in \Ocal\mid\phi(x)\geq1/k\}$ would be countable.) For each $M\geq1$, let $\Pi$ be a Taibleson partition with the property that, for each $1\leq m\leq M$, each of the points $x_1,\dots,x_M$ are in distinct discs of the partition. Then $V_\Pi(f)\geq\epsilon M$ and hence $V_\Taib(f)=+\infty$.  
\end{proof}

\begin{ex}\label{InfiniteTaibVariation}
This example shows that the converse of Proposition~\ref{TaibContinuousProp} is false, by constructing a continuous function $f:\Ocal\to\RR$ with $V_\Taib(f)=+\infty$.  Let $\{c_m\}_{m=1}^{\infty}$ be an infinite sequence of nonzero points in $\Ocal$ with $c_m\to0$, and let $\{D_m\}_{m=1}^{\infty}$ be a sequence of disjoint discs in $\Ocal$ with $c_m\in D_m$ and $0\not\in D_m$ for all $m$.  Let $f:\Ocal\to\RR$ be a function which is locally constant on each disc $D_m$, and which takes exactly two values on each $D_m$, the values $0$ and $1/m$.  Next define $f(x)=0$ for all $x\in\Ocal\setminus(\cup_mD_m)$.  Then for each $M\geq1$ one can construct a Taibleson partition $\Pi$ containing the discs $D_1,\dots,D_M$ (and some other discs), and 
\[
V_\Pi(f)\geq1+1/2+1/3+\dots+1/M.
\]
It follows that $V_\Taib(f)=+\infty$ by the divergence of the harmonic series.
\end{ex}

Despite the previous example, the following partial converse of Proposition~\ref{TaibContinuousProp} for Lipschitz functions is available.

\begin{prop}\label{TaibTipContinuousProp}
If a function $f:\Ocal\to\RR$ is Lipschitz continuous with constant $C\geq0$, then $V_\Taib(f)\leq C$.
\end{prop}
\begin{proof}
By hypothesis $|f(x)-f(x)|\leq C|x-y|$ for all $x,y\in\Ocal$.  If $D\subseteq\Ocal$ is a disc, then its radius is $\mu(D)$, and so $\sup_{x,y\in D}|f(x)-f(y)|\leq C\mu(D)$.  It follows that if $\Pi$ is any Taibleson partition of $\Ocal$, we have $V_\Pi(f)\leq C$ because $\sum_{D\in\Pi}\mu(D)=\mu(\Ocal)=1$.  We conclude that $V_\Taib(f)\leq C$.
\end{proof}

\begin{ex}\label{AlternatingExample}
In this example we construct a function $f:\Ocal\to\RR$ which is continuous, but not Lipschitz continuous, and which has finite Taibleson variation.  

For each $k\geq1$, let $A_k=D_{1/q^{k+1}}(\pi^k)$.  Every $x\in A_k$ satisfies $|x|=1/q^k$ and hence the $A_k$ are pairwise disjoint.  Define a function $f:\Ocal\to\RR$ by
\begin{equation}\label{AlternatingExampleFunction}
f(x)=\sum_{k\geq0}\frac{(-1)^{k}}{k+1}\Xcal_{A_k}(x).
\end{equation}

We show that $V_\Taib(f)=1$.  To see this, we first claim that if $D\subseteq\Ocal$ is any disc which does not contain $0$, then $f$ is constant on $D$.  This is clearly the case if $D$ is either contained in some $A_k$ or is disjoint from all of the $A_k$, so the only case left to check is when $D$ properly contains some $A_k$.  But if $D$ properly contains $A_k$, then the radius of $D$ is at least $1/q^k$ and $D$ contains $\pi^k$, whereby $0\in D_{1/q^k}(\pi^k)\subseteq D$, contradicting the asumption that $0\not\in D$.

Since $f$ is constant on any disc which does not contain $0$, such discs cannot contribute to the variation $V_\Pi(f)$ associated to any Taibleson partition $\Pi$ of $\Ocal$.  If $\Pi$ is a Taibleson partition of $\Ocal$ and $D_0\in\Pi$ denotes the disc containing $0$, then 
\[
V_\Pi(f)=\sup_{x,y\in D_0}|f(x)-f(y)|\leq 1,
\]
because all partial sums of the alternating harmonic series $\sum_{k\geq0}(-1)^{k}/(k+1)$ are in the interval $[0,1]$.  Choosing $\Pi=\{\Ocal\}$ shows that $V_\Pi(f)=1$ can be achieved, and therefore $V_\Taib(f)=1$.

We note that $f$ is locally constant except at $0$, and it is continuous at $0$ by the convergence of the alternating harmonic series.  To see that $f$ is not Lipschitz continuous, if $|x|=1/q^m$ then $|f(0)-f(x)|$ is the error term in the alternating harmonic series $\sum_{0\leq k\leq m}(-1)^{k}/(k+1)$, which is $\approx\frac{1}{m}$, and thus no bound of the form $|f(0)-f(x)|\leq C|x|=C/q^m$ is possible.
\end{ex}

In the previous example, the fact that discs not containing zero cannot contribute to the Taibleson variation leaves one with the impression that Taibleson variation is too coarse to detect the kind of oscillation exhibited by alternating sums of the type (\ref{AlternatingExampleFunction}).  Indeed, in $\S$~\ref{BerkVarSection} we will show that $f$ has infinite Berkovich-analytic variation.  

Our final result in this section is the following theorem, which shows that it is not possible to prove a Koksma inequality of the form (\ref{KoksmaGeneral}) in which the constant $C(f)$ depends only on the Taibleson variation $V_\Taib(f)$.  Once again the culprit is the lack of sensitivity of Taibleson variation to oscillation.

\begin{thm}\label{NoKoksmaThm}
For each integer $M\geq1$, there exists a locally constant function $f:\Ocal\to\RR$ and a finite subset $X$ of $\Ocal$ such that $V_\Taib(f) = 2$ and 
\[
\left|\frac{1}{|X|}\sum_{x\in X}f(x)-\int_\Ocal f\,d\mu\right| \geq 2M\Delta(X)
\]
\end{thm}

\begin{proof}
For each $k\geq 0$, let 
\begin{equation*}
\begin{split}
A_k = D_{1/q^{k+1}}(\pi^k) = \{\pi^k+a_{k+1}\pi^{k+1}+a_{k+2}\pi^{k+2}+\dots\mid a_i\in S\} 
\end{split}
\end{equation*}
Every $x\in A_k$ satisfies $|x|=1/q^k$ and hence the $A_k$ are pairwise disjoint.  Define $f:\Ocal\to\RR$ by
\[
f(x) = \sum_{0\leq k\leq 2M-1}(-1)^{k}\Xcal_{A_{k}}(x).
\]
Note that $f$ is a slight variant of the function considered in Example~\ref{AlternatingExample}.  By the same argument used in Example~\ref{AlternatingExample}, because $f$ is constant on any disc which does not contain $0$, we have $V_\Taib(f)=2$.  

Our strategy is to construct a set $X$ which has very small discrepancy, except that $X\cap A_k$ has one ``extra'' point when $k$ is even, and one ``missing'' point when $k$ is odd.  In this way the set $X$ will take advantage of the oscillation built into the function $f$.  Let $T\geq 2M$ be an integer, and define
\[
Y=\{a_0+a_1\pi+\dots+a_{T-1}\pi^{T-1}\in\Ocal\mid a_i\in S\},
\]
thus $|Y|=q^T$.  Next define
\[
X=Y\cup\{1+\pi^T,\pi^2+\pi^T,\dots,\pi^{2M-2}+\pi^T\}\setminus\{\pi,\pi^3,\dots,\pi^{2M-1}\}.
\]
Thus $|X|=|Y|=q^T$, but $X$ differs from $Y$ in that one point $\pi^k+\pi^T$ has been added to $X$ in each of the $M$ discs $A_0,A_2,\dots,A_{2M-2}$ with even indices, and one point $\pi^k$ has been removed from $X$ in each of the $M$ discs $A_1,A_3,\dots,A_{2M-1}$ with odd indices.  Explicitly, 
\begin{equation}\label{AlternatingSetSizes}
|X\cap A_k|=q^{T-(k+1)}+(-1)^k.
\end{equation}

We will show that $\Delta(X)\leq 1/q^T$.  Thus we must show that 
\begin{equation}\label{CounterExampleDiscrepBound}
\left|\frac{|X\cap D|}{|X|}-\mu(D)\right| \leq 1/q^T
\end{equation}
for all discs $D\subseteq \Ocal$.  Suppose that $D$ has radius $\mu(D)=1/q^n$.

We first consider the case $n\geq T+1$. Then $D$ can contain at most one point of $X$, as any two distinct points $x,x'\in X$ satisfy $|x-x'|\geq 1/q^{T}>1/q^n$.  Thus $|X\cap D|=\theta$ where $\theta$ is either $0$ or $1$, and $\frac{|X\cap D|}{|X|}-\mu(D)=\frac{\theta-q^{T-n}}{q^T}$, and (\ref{CounterExampleDiscrepBound}) holds.  If instead $n=T$, then $D$ can contain at most two points of $X$, and $D$ contains two points of $X$ only when these two points are a pair $\pi^k$ and $\pi^k+\pi^T$ for even $k$.  Thus $|X\cap D|=\theta$ where $\theta$ is either $0,1$ or $2$, and $\frac{|X\cap D|}{|X|}-\mu(D)=\frac{\theta-1}{q^T}$, and (\ref{CounterExampleDiscrepBound}) holds.

Finally we consider the case $0\leq n<T$.  By an argument described in Example~\ref{AlternatingExample}, at least one of the following cases must hold: either (i) $D\subseteq A_k$ for some $k=0,1,2,\dots,2M-1$, or (ii) $D$ is disjoint from all of the discs $A_0,A_1,A_2,\dots,A_{2M-1}$, or (iii) $D$ contains $0$.

First, if $D\subseteq A_k$ for some $k=0,1,2,\dots,2M-1$, then $|X\cap D|=q^{T-n}+\theta$, where $\theta\in\{-1,0,1\}$.  Note that $\theta=1$ can only occur when $k$ is even, due to the presence of the ``extra point'' $\pi^k+\pi^T$ in $A_k$, while $\theta=-1$ can only occur when $k$ is odd, due to the absence of the point $\pi^k$, which has been removed from $X$.  We then have $\frac{|X\cap D|}{|X|}-\mu(D)=\frac{q^{T-n}+\theta}{q^T}-\frac{1}{q^n}=\frac{\theta}{q^T}$, and (\ref{CounterExampleDiscrepBound}) holds.  If $D$ is disjoint from all of the discs $A_0,A_1,A_2,\dots,A_{2M-1}$, then $|X\cap D|=q^{T-n}$ and (\ref{CounterExampleDiscrepBound}) holds in this case as well.  

Finally, assume that $D$ contains $0$, thus $D=D_{1/q^n}(0)$.  If $n>2M-1$ then $D$ is disjoint from all of the discs $A_0,A_1,A_2,\dots,A_{2M-1}$, a case which has already been treated.  If $0\leq n\leq 2M-1$, then $D$ contains the discs $A_n,A_{n+1},\dots,A_{2M-1}$ and is disjoint from the discs $A_0,A_{1},\dots,A_{n-1}$.  Therefore $|X\cap D| = q^{T-n}+\theta$, where $\theta$ is the number of even indices $k\in[n,2M-1]$ minus the number of odd indices $k\in[n,2M-1]$.  Thus $\theta\in\{-1,0\}$ and once again (\ref{CounterExampleDiscrepBound}) holds, completing the proof that $\Delta(X)\leq 1/q^T$.

Finally, we use (\ref{AlternatingSetSizes}) to calculate 
\begin{equation*}
\begin{split}
\frac{1}{|X|}\sum_{x\in X}f(x)-\int_\Ocal f\,d\mu & = \frac{1}{q^T}\sum_{k=0}^{2M-1}(-1)^{k}|X\cap A_k|-\sum_{k=0}^{2M-1}(-1)^{k}\mu(A_k) \\
	& = \frac{1}{q^T}\sum_{k=0}^{2M-1}(-1)^{k}(q^{T-(k+1)}+(-1)^k)-\sum_{k=0}^{2M-1}(-1)^{k}q^{-(k+1)} \\
	& = 2M/q^T \\
	& \geq 2M\Delta(X), 
\end{split}
\end{equation*}
completing the proof.
\end{proof}

Inspection of the proof of Theorem~\ref{NoKoksmaThm} shows that for arbitrary $M\geq1$, the sets $X$ satisfying the conclusion can be found with  arbitrarily small discrepancy $\Delta(X)$.  This is notable because the case of greatest interest in any Koksma inequality is the case of a sequence $\{X_n\}$ of equidistributed sets.

\section{Beer variation}\label{BeerVarSection}

In this section we describe Beer's notion \cite{MR237441} of the variation of a function $f:\Ocal\to\RR$.  Recall from $\S$~\ref{PrelimSection} that, given any complete set $S$ of coset representatives in $\Ocal$ for the residue field $\Ocal/\pi \Ocal$, with $0\in S$, each element $x\in \Ocal$ can be written uniquely as a power series in $\pi$ with coefficients in $S$.  In order to define Beer variation, we need to fix once and for all an order on the set $S$, and we do so by indexing the elements of $S$ as
\[
S=\{s_0,s_1,s_2,\dots,s_{q-1}\}.
\] 
Here we are generalizing the construction of Beer, who formulated her definition only in the case $\Ocal=\ZZ_p$ and $S=\{0,1,2,\dots,p-1\}$.

Fix a (large) positive integer $\lambda$, and let $m_0,m_1,m_2,\dots,m_{q^\lambda-1}$ be the list of all $q^\lambda$ elements of $\Ocal$ of the form 
\[
m_i=a_0+a_1\pi+a_2\pi^2+\dots+a_{\lambda-1}\pi^{\lambda-1},
\]
written in the dictionary ordering according to the coefficients $a_0,a_1,a_2,\dots,a_{\lambda-1}\in S$ and with respect to the ordering of the set $S$.  For each $i = 0,1,2,\dots,q^{\lambda}-1$, let $E_i = D_{1/q^{\lambda}}(m_i)$.  We call the collection of discs $E_0, E_1, E_2, \dots, E_{q^\lambda-1}$ the {\bf ordered Beer partition} of $\Ocal$ associated to $S$ and $\lambda$.  Given a function $f:\Ocal\to \RR$, let
\begin{equation} 
\label{eqn: vlambdaf}
V_{\lambda}(f) = \sup_{x_i\in E_i} \sum_{i=1}^{q^\lambda - 1} |f(x_i)-f(x_{i-1})|.
\end{equation} 
It is easy to see that $V_\lambda(f)$ is monotone increasing as a sequence in $\lambda$, since the ordered Beer partition associated to $\lambda+1$ is finer than that associated to $\lambda$.  The {\bf Beer variation} of $f$ is defined as
\begin{equation}
\label{eqn: Beer-bounded-variation}
V_\Beer(f) =  \lim_{\lambda\to+\infty} V_{\lambda}(f).
\end{equation}


The motivation for the dictionary ordering considered by Beer can be explained as follows. Given any $0\leq k \leq \lambda$ and $a_0,a_1,\dots,a_{k-1}\in S$, the disc with center $a = a_0 + a_1\pi +\dots + a_{k-1}\pi^{k-1}$ and radius $1/q^k$ in $\Ocal$ can be written as $D_{1/q^k}(a) = a + \pi^k \Ocal$.  Thus $D_{1/q^k}(a)$ is the set of all elements $x \in \Ocal$ whose first $k$ terms in its $\pi$-adic expansion begin with $a_0 + a_1\pi + a_2\pi^2 +\dots + a_{k-1}\pi^{k-1}$.  We conclude that any disc in $\Ocal$ is the union of a block of {\bf consecutive} discs $E_i$ in the dictionary ordered partition $E_0,E_1,E_2,\dots, E_{q^\lambda-1}$ of $\Ocal$.  

We observe that the point $\alpha=\sum_{k\geq0}s_0\pi^k=s_0/(1-\pi)$ is contained in the first disc $E_0$ of the ordered partition for all $\lambda\geq1$, and similarly $\beta=\sum_{k\geq0}s_{q-1}\pi^k=s_{q-1}/(1-\pi)$ is always contained in the last disc $E_{q^\lambda-1}$.  Consequently, as the following examples show, the behavior of a function $f:\Ocal\to \RR$ at the points $\alpha$ and $\beta$ has a strong influence on the evaluation of the Beer variation (in much the same way that the endpoints effect the classical real valuation of a function $f:[0,1]\to\RR$).  In particular, Beer variation depends on the choice of the ordered set $S$ of coset representatives for the quotient $\Ocal/\pi \Ocal$.

\begin{ex}
Consider the characteristic function $f=\Xcal_{A}:\Ocal\to\RR$ of a proper subdisc $A\subsetneq \Ocal$. We show that
\begin{equation} 
\label{eqn: Beer-variation-disc}
V_\Beer(f) = \begin{cases}
 1  & \text{ if either $\alpha\in A$ or $\beta\in A$} \\
 2  & \text{ if $\alpha\not\in A$ and $\beta\not\in A$.}
  \end{cases}  
\end{equation} 
(Note that the disc $A$ cannot contain both $\alpha$ and $\beta$ because of the assumption that $A\neq \Ocal$ together with the fact that $|\alpha-\beta|=1$.)  Fix $\lambda \geq 1$ large enough so that the radius of $A$ is $\geq1/q^{\lambda}$.  Then $A$ is the union of a block of consecutive discs $E_i$ in the ordered partition, but the block does not contain all of the discs $E_i$ because $A\neq \Ocal$.  If $\alpha\in A$ then this is an initial block $E_0,E_1,\dots E_n$, and it follows that $V_{\lambda}(f) = 1$.  Similarly, $V_{\lambda}(f) = 1$ if $\beta\in A$.  If neither $\alpha$ nor $\beta$ is in $A$, then $A$ is the union of a block of discs $E_i$ containing neither $E_0$ nor $E_{q^\lambda -1}$. Therefore, $V_{\lambda}(f) = 2$, and (\ref{eqn: Beer-variation-disc}) follows. 

This example is similar to the case of the classical real variation of the characteristic function of a proper subinterval $[a,b]$ of $[0, 1]$, which is equal to $1$ if $a=0$ or $b=1$, and is equal to $2$ if $0<a<b<1$.
\end{ex}

\begin{ex}\label{BeerPoweringExample}
Let $c\in \Ocal$, let $t>0$, and consider the function $f(x) = |x-c|^t$.  We show that 
\[
V_\Beer(f)=|\alpha-c|^t+|\beta-c|^t.
\]
In particular $1\leq V_\Beer(f) \leq2$, with $V_\Beer(f)=1$ if and only if $c=\alpha$ or $c=\beta$.

We first consider the case that $c=\alpha$.  Note that $f(x)=|x-\alpha|^t$ is constant on every disc in $\Ocal$ that does not contain $\alpha$.  For fixed large $\lambda$, in the ordered partition $E_0,E_1,\dots,E_{q^\lambda-1}$ associated to $\lambda$, we can group the discs $E_i$ into $\lambda+1$ blocks, where $f(x)=|x-\alpha|^t$ is constant on each block, except the first block which contains only $E_0$. Thus the sum occurring in the definition of $V_\Beer(f)$ given by (\ref{eqn: vlambdaf}) is maximized when $x_0=\alpha$, and the choices of $x_1, x_2,\dots,x_{q^\lambda -1}$ in their respective discs $E_i$ are arbitrary. Since $0=f(x_0)\leq f(x_1)\leq\dots\leq f(x_{q^\lambda-1})=1$, it follows that
\begin{equation*}
V_\lambda(f) = \sum_{i=1}^{q^\lambda -1} |f(x_i) - f(x_{i-1})| =1.
\end{equation*} 
and we conclude that $V_\Beer(f) = 1$.  The proof in the case $c=\beta$ is similar.

Assume that $c\neq\alpha$ and $c\neq\beta$.  Then $f(x)=|x-c|^t$ takes a constant value $C_i$ on each partition disc $E_i$ except the one containing $c$, call it $E_{i_0}$.  The sequence $C_0,C_1,\dots,C_{i_0-1}$ starts at $|\alpha-c|^t$ and is monotone decreasing, the sequence $C_{i_0+1},C_{i_0+2},\dots,C_{q^\lambda-1}$ is monotone increasing and ends at $|\beta-c|^t$, and thus the sum occurring in the definition of $V_\Beer(f)$ given by (\ref{eqn: vlambdaf}) is maximized when $x_{i_0}=0$, and the choices of the other $x_i$ in their respective discs $E_i$ are arbitrary.  A calculation similar to the one above shows that $V_\Beer(f)=|\alpha-c|^t+|\beta-c|^t$.

Again we point out the similarity of this example to the classical real valuation of the function $f:[0,1]\to\RR$ defined by $f(x)=|x-c|^t$ for $c\in[0,1]$, which is equal to $|c|^t+|1-c|^t$.
\end{ex}

\begin{prop}\label{TaiblesonLeqBeerProp}
For any $f: \Ocal \to \RR$, we have $V_\Taib(f) \leq V_\Beer(f)$.
\end{prop}
\begin{proof}
Let $\Pi$ be a Taibleson partition of $\Ocal$. For each disc $D\in\Pi$, consider arbitrary $\alpha_D,\beta_D\in D$ with $\alpha_D\neq \beta_D$, and without loss of generality assume that $f(\beta_D)\leq f(\alpha_D)$.

We produce a corresponding Beer partition as follows. Take $\lambda \geq 1$ large enough so that al of the discs in the Taibleson partition has radius at least $1/q^\lambda$, and $\alpha_D\not\equiv\beta_D \pmod{\pi^\lambda}$ for all $D\in\Pi$. Let $E_0,E_1,\dots ,E_{q^\lambda-1}$ be the ordered Beer partition of $\Ocal$ associated to $\lambda$.  

For each disc $D\in\Pi$, by our choice of $\lambda$ we know that $D$ is a union of consecutive discs $E_s, E_{s+1},\dots,E_{t}$ in the Beer partition.  The point $\alpha_D$ is in one of these discs, say $E_a$, and $\beta_D$ is in another disc, say $E_b$. Select $x_a = \alpha_k$, $x_b = \beta_k$, and make arbitrary choices for the $x_i \in E_i$ with $i \neq a$ and $i \neq b$. Presuming $a > b$, we have   
\begin{equation*}
\begin{split}
 f(\alpha_D)  - f(\beta_D) & =  f(x_a) - f(x_b) \\
  & = \sum_{b < i \leq a}  f(x_{i}) - f(x_{i-1}) \\
  & \leq \sum_{b < i \leq a} | f(x_{i}) - f(x_{i-1}) | \\
  & \leq \sum_{s< i \leq t} | f(x_{i}) - f(x_{i-1}) |.
\end{split}
\end{equation*}
If instead $b < a$, we still have
\[
 f(\alpha_D)  - f(\beta_D)  \leq \sum_{s< i \leq t} | f(x_{i}) - f(x_{i-1}) |
\]
by a similar argument. 

Now, summing over all discs in the Taibleson partition we obtain 
\begin{equation*}
\begin{split}
\sum_{D\in\Pi}(f(\alpha_D)  - f(\beta_D)) & \leq \sum_{1 < i \leq q^\lambda -1} |f(x_{i}) - f(x_{i-1})| \\
	& \leq V_\lambda(f) \\ 
	& \leq V_\Beer(f). 
\end{split}
\end{equation*}
Taking the supremum over all choices of $\alpha_D,\beta_D\in D$ for all $D\in\Pi$, we obtain $V_\Pi(f)\leq V_\Beer(f)$, and finally taking the supremum over all Taibleson partitions $\Pi$, we conclude that $V_\Taib(f) \leq V_\Beer(f)$.
\end{proof}

Combining Proposition~\ref{TaiblesonLeqBeerProp} with Proposition~\ref{TaibContinuousProp} we obtain the following regularity property for functions with finite Beer variation.

\begin{prop}\label{BeerContinuousProp}
If a function $f:\Ocal\to\RR$ has finite Beer variation, then there exists a countable subset $Z$ of $\Ocal$ such that $f$ is continuous at every point in $\Ocal\setminus Z$.
\end{prop}

On the other hand, it also follows from Proposition~\ref{TaiblesonLeqBeerProp} and Example~\ref{InfiniteTaibVariation} that a continuous function need not have finite Beer variation.

Beer used her notion of variation to prove a $p$-adic Koksma inequality, as we now describe.  For the purposes of Beer's theorem, we say a function $f:\Ocal\to\RR$ is {\bf integrable} if there exists sequences $u_n,v_n:\Ocal\to\RR$ of functions (for $n\geq1$), each defined as a finite linear combination of characteristic functions of discs, with $u_n\leq f\leq v_n$ and $\int_\Ocal(v_n-u_n)d\mu\to0$. This notion is comparable to Riemann integrability in real analysis.  As might be expected, the class of integrable functions $f:\Ocal\to\RR$ in the sense of Beer contains all continuous functions, but is strictly smaller than the class of all (measure-theoretic) Haar-integrable functions.

\begin{thm}[Beer \cite{MR237441}]
\label{thm: Beer-Koksma}
If $f: \Ocal \to \RR$ is an integrable function and $X$ is a finite subset of $\Ocal$ with discrepancy $\Delta(X)$, then  
\[
\left| \frac{1}{N} \sum_{x\in X} f(x) -\int_{\Ocal} f \, d\mu\right|  \leq 2qV_\Beer(f) \Delta(X).
\]
\end{thm} 

One way of viewing Beer's proof of this result is to note that, using the dictionary ordering on $\Ocal$, one can essentially identify $\Ocal$ with the unit interval $[0,1]$ of the real line, using base $q$ expansions of real numbers. Then the proof of Theorem~\ref{thm: Beer-Koksma} follows precisely the same argument as the proof of the classical Koksma inequality for a real interval.  We refer the reader to \cite{MR237441} for details.

However, Beer's approach has a notable shortcoming.  In order to reduce the proof of Theorem~\ref{thm: Beer-Koksma} to the proof of the classical real Koksma inequality, one needs to make an arbitrary choice of an ordering on $\Ocal$, and the value of $V_\Beer(f)$ depends on this choice.  Moreover, Beer variation is not translation invariant; that is, when $f:\Ocal\to\RR$ and $c\in\Ocal$, it is not necessarily the case that $f(x)$ and $f(x-c)$ have the same Beer variation.  This is unfortunate, since $\Ocal$ is a group.  

\section{Berkovich-analytic variation}\label{BerkVarSection}

In view of the previous two sections, it would be desirable to define a notion of variation on $\Ocal$ which is order-free and translation invariant (like Taibleson variation), but which admits a Koksma inequality (like Beer variation).  We define such a variation in this section.

Let $\Dcal_\Ocal$ denote the collection of all subdiscs $D$ of $\Ocal$.  Define a relation $\prec$ on $\Dcal_\Ocal$, declaring that $D'\prec D$ whenever $D'\subseteq D$ and $\mu(D')=\frac{1}{q}\mu(D)$.  Thus for each disc $D\in\Dcal_\Ocal$, there are precisely $q$ discs $D'$ satisfying $D'\prec D$, and these discs form a partition of $D$. 

Let $f:\Ocal\to\RR$ be a Haar-integrable function.  For each disc $D\in\Dcal_\Ocal$, we define
\[
f(D)=\frac{1}{\mu(D)}\int_D f\, d\mu,
\]
the average value taken by $f$ on the disc $D$.  In this way we have extended $f:\Ocal\to\RR$ to a function $f:\Ocal\amalg\Dcal_\Ocal\to\RR$; in a slight abuse of notation we use $f$ to denote both functions.  We now define the {\bf Berkovich-analytic variation} of $f:\Ocal\to\RR$ by
\begin{equation}\label{BAVarDef}
V_{\Berk}(f) = \sum_{D\in\Dcal_\Ocal}\sum_{D'\prec D}|f(D')-f(D)|.
\end{equation}
Thus $V_{\Berk}(f)$ records the absolute differences between the average value of $f$ on a disc $D$ and all of its $q$ subdiscs $D'$ with $D'\prec D$, and sums this amount over all subdiscs $D$ of $\Ocal$.

It is conceptually helpful to interpret $V_\Berk(f)$ as the variation of $f$ when it is naturally extended to a function on a certain infinite subtree of the Berkovich affine line over $K$.  To understand this interpretation, we define an infinite tree $T_\Ocal$ associated to the compact local ring $\Ocal$, as follows.  The vertices of $T_\Ocal$ are in bijective correspondence with the discs $D\subseteq\Ocal$; abusing notation slightly we also denote by $D$ the vertex of $T_\Ocal$ associated to the disc $D$.  We declare that two vertices $D'$ and $D$ of $T_\Ocal$ are connected by an edge if $D'\prec D$ as discs.  Thus $T_\Ocal$ is a complete $q$-ary rooted tree.  The root vertex $\Ocal$ meets $q$ edges, sharing one with each of its children vertices $D'\prec \Ocal$ of radius $1/q$.  Each non-root vertex $D$ shares one edge with its parent vertex and $q$ edges with its children vertices.  

If we define $\overline{T}_\Ocal=\Ocal\amalg T_\Ocal$, the disjoint union of $\Ocal$ and the rooted tree $T_\Ocal$, then $\overline{T}_\Ocal$ can be naturally identified with a subset of the Berkovich affine line $\AA^1_{\Berk,\KK}$ over the completion $\KK$ of the algebraic closure of $K$, as described say in Berkovich \cite{MR1070709} or Baker-Rumely \cite{MR2599526}.  Alternatively, $T_\Ocal$ may be identified with a subset of the Bruhat-Tits tree associated to $\mathrm{PGL}_2(\Ocal)$, see \cite{MR2482346}.

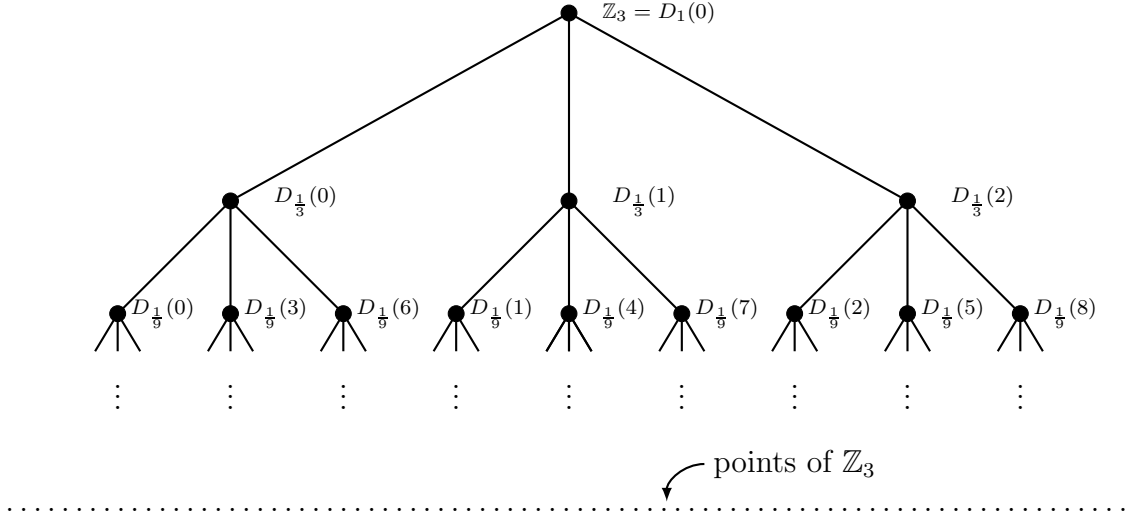
\begin{figure}

\begin{tikzpicture}

\draw[fill=black] (0,4) circle (3pt);
\draw[fill=black] (-4.5,1.5) circle (3pt);
\draw[fill=black] (0,1.5) circle (3pt);
\draw[fill=black] (4.5,1.5) circle (3pt);
\draw[fill=black] (-6,0) circle (3pt);
\draw[fill=black] (-4.5,0) circle (3pt);
\draw[fill=black] (-3,0) circle (3pt);
\draw[fill=black] (-1.5,0) circle (3pt);
\draw[fill=black] (0,0) circle (3pt);
\draw[fill=black] (1.5,0) circle (3pt);
\draw[fill=black] (3,0) circle (3pt);
\draw[fill=black] (4.5,0) circle (3pt);
\draw[fill=black] (6,0) circle (3pt);

\node at (1.2,4) {\tiny{$\ZZ_3=D_1(0)$}};
\node at  (-3.5,1.5) {\tiny{$D_{\frac{1}{3}}(0)$}};
\node at  (1,1.5) {\tiny{$D_{\frac{1}{3}}(1)$}};
\node at  (5.5,1.5) {\tiny{$D_{\frac{1}{3}}(2)$}};
\node at  (-5.4,0) {\tiny{$D_{\frac{1}{9}}(0)$}};
\node at  (-3.9,0) {\tiny{$D_{\frac{1}{9}}(3)$}};
\node at  (-2.4,0) {\tiny{$D_{\frac{1}{9}}(6)$}};
\node at  (-.9,0) {\tiny{$D_{\frac{1}{9}}(1)$}};
\node at  (0.6,0) {\tiny{$D_{\frac{1}{9}}(4)$}};
\node at  (2.1,0) {\tiny{$D_{\frac{1}{9}}(7)$}};
\node at  (3.6,0) {\tiny{$D_{\frac{1}{9}}(2)$}};
\node at  (5.1,0) {\tiny{$D_{\frac{1}{9}}(5)$}};
\node at  (6.6,0) {\tiny{$D_{\frac{1}{9}}(8)$}};

\draw[thick] (0,4) -- (-4.5,1.5);
\draw[thick] (0,4) -- (0,1.5);
\draw[thick] (0,4) -- (4.5,1.5);
\draw[thick] (-4.5,1.5) -- (-6,0);
\draw[thick] (-4.5,1.5) -- (-4.5,0);
\draw[thick] (-4.5,1.5) -- (-3,0);
\draw[thick] (0,1.5) -- (-1.5,0);
\draw[thick] (0,1.5) -- (0,0);
\draw[thick] (0,1.5) -- (1.5,0);
\draw[thick] (4.5,1.5) -- (3,0);
\draw[thick] (4.5,1.5) -- (4.5,0);
\draw[thick] (4.5,1.5) -- (6,0);

\draw[thick] (0,0) -- (0,-.5);
\draw[thick] (0,0) -- (-.3,-.5);
\draw[thick] (0,0) -- (.3,-.5);

\draw[thick] (0,0) -- (0,-.5);
\draw[thick] (0,0) -- (-.3,-.5);
\draw[thick] (0,0) -- (.3,-.5);

\draw[thick] (1.5,0) -- (1.5,-.5);
\draw[thick] (1.5,0) -- (1.2,-.5);
\draw[thick] (1.5,0) -- (1.8,-.5);

\draw[thick] (3,0) -- (3,-.5);
\draw[thick] (3,0) -- (2.7,-.5);
\draw[thick] (3,0) -- (3.3,-.5);

\draw[thick] (4.5,0) -- (4.5,-.5);
\draw[thick] (4.5,0) -- (4.2,-.5);
\draw[thick] (4.5,0) -- (4.8,-.5);

\draw[thick] (6,0) -- (6,-.5);
\draw[thick] (6,0) -- (5.7,-.5);
\draw[thick] (6,0) -- (6.3,-.5);

\draw[thick] (-1.5,0) -- (-1.5,-.5);
\draw[thick] (-1.5,0) -- (-1.8,-.5);
\draw[thick] (-1.5,0) -- (-1.2,-.5);

\draw[thick] (-3,0) -- (-3,-.5);
\draw[thick] (-3,0) -- (-3.3,-.5);
\draw[thick] (-3,0) -- (-2.7,-.5);

\draw[thick] (-4.5,0) -- (-4.5,-.5);
\draw[thick] (-4.5,0) -- (-4.8,-.5);
\draw[thick] (-4.5,0) -- (-4.2,-.5);

\draw[thick] (-6,0) -- (-6,-.5);
\draw[thick] (-6,0) -- (-5.7,-.5);
\draw[thick] (-6,0) -- (-6.3,-.5);

\node at  (-6,-1) {$\vdots$};
\node at  (-4.5,-1) {$\vdots$};
\node at  (-3,-1) {$\vdots$};
\node at  (-1.5,-1) {$\vdots$};
\node at  (0,-1) {$\vdots$};
\node at  (1.5,-1) {$\vdots$};
\node at  (3,-1) {$\vdots$};
\node at  (4.5,-1) {$\vdots$};
\node at  (6,-1) {$\vdots$};

\node at  (3,-2) {points of $\ZZ_3$};

\draw[-latex, thick] ($(1.8,-2)$) arc
    [
        start angle=90,
        end angle=180,
        x radius=0.5cm,
        y radius =0.5cm
    ]; 
    
\node at  (0,-2.6) {$\dots\dots\dots\dots\dots\dots\dots\dots\dots\dots\dots\dots\dots\dots\dots\dots\dots\dots\dots\dots\dots\dots\dots\dots\dots\dots\dots$};

\end{tikzpicture}

  \caption{The extended infinite rooted tree $\overline{T}_\Ocal$ when $\Ocal=\ZZ_3$, with detail on the first three levels showing vertices associated to discs of radius $1$, $1/3$, and $1/9$.} 
\end{figure}

If $f:\Ocal\to\RR$ is Haar-integrable, we extend $f$ to a function $f:\overline{T}_\Ocal\to\RR$ as follows.  For each vertex $D$ of  $T_\Ocal$, define $f(D)=\mu(D)^{-1}\int_Df\,d\mu$, as in the definition of Berkovich-analytic variation $V_\Berk(f)$.  On each edge of the tree $T_\Ocal$, say an edge connecting vertices $D'$ and $D$ for $D'\prec D$, we extend $f$ linearly given the knowledge of the values $f(D')$ and $f(D)$.  We may now reinterpret the Berkovich analytic variation defined in (\ref{BAVarDef}) as a graph theoretic variation
\[
V_\Berk(f) = \sum_{E}\sup_{x,y\in E}|f(x)-f(y)|,
\]
the sum taken over all edges $E$ in the tree $\overline{T}_\Ocal$.  Note that in our case, it has been mandated that $f$ is linear on each edge $E$, and hence $\sup_{x,y\in E}|f(x)-f(y)|$ is always achieved by choosing $x$ and $y$ to be the two endpoints of $E$.

\begin{ex}
Consider the characteristic function $f=\Xcal_{A}:\Ocal\to\RR$ of a proper subdisc $A\subsetneq \Ocal$  We show that
\begin{equation}\label{BerkDiscrepDisc1}
V_\Berk(f)=2(1-\mu(A))
\end{equation}
In particular $2-\frac{2}{q}< V_\Berk(f)< 2$.

We first observe that, if a disc $D\subseteq \Ocal$ does not properly contain $A$, then $f$ is constant on $D$ and hence $|f(D')-f(D)|=0$ for all $D'\prec D$.  Assuming that $A$ has radius $\mu(A)=1/q^k$, we may label as
\[
A=D_k\prec D_{k-1}\prec\dots \prec D_1\prec D_0=\Ocal
\]
the ascending chain of discs containing $A$, where $D_m$ has radius $1/q^m$.  Thus
\begin{equation}\label{BerkDiscrepDisc2}
\begin{split}
V_\Berk(f) & = \sum_{0\leq m\leq k-1} \sum_{D'\prec D_m}|f(D')-f(D_m)|.
\end{split}
\end{equation}
If $0\leq m\leq k-1$, then $f(D_{m})=q^{m-k}$, $f(D_{m+1})=q^{m-k+1}$, and $f(D')=0$ for the remaining $q-1$ discs $D'\prec D_{m}$.  Thus
\begin{equation}\label{BerkDiscrepDisc3}
\sum_{D'\prec D_{m}}|f(D')-f(D_m)|=q^{m-k+1}-q^{m-k} + (q-1)q^{m-k}.
\end{equation}
We obtain (\ref{BerkDiscrepDisc1}) from (\ref{BerkDiscrepDisc2}) and (\ref{BerkDiscrepDisc3}) using a gemetric series calculation.
\end{ex}

\begin{ex}\label{BerkPoweringExample}
Let $c\in \Ocal$, let $t>0$, and consider the function $f:\Ocal\to\RR$ defined by $f(x) = |x-c|^t$.  We will show that 
\[
V_\Berk(f) =\frac{2(q-1)}{q-q^{-t}}.  
\]
In particular $1<V_\Berk(f)<2$.

By translation invariance we may assume without loss of generality that $c=0$ and thus $f(x)=|x|^t$.  If a disc $D\subseteq \Ocal$ does not contain zero, then $f$ is constant on $D$ and hence $\sum_{D'\prec D}|f(D')-f(D)|=0$.  Thus letting $D_n=D_{1/q^n}(0)$, we have
\begin{equation*}
V_\Berk(f) = \sum_{n\geq0}\sum_{D'\prec D_n}|f(D')-f(D_n)|.
\end{equation*} 
We calculate
\begin{equation*}
\begin{split}
f(D_n) & = \frac{1}{\mu(D_n)}\int_{D_n}|x|^td\mu(x) \\
	& = q^n\sum_{k\geq n}\mu(D_k\setminus D_{k+1})q^{-kt} \\
	& = q^n\sum_{k\geq n}(q^{-k}-q^{-k-1})q^{-kt} \\
	& = Cq^{-nt}
\end{split}
\end{equation*}
via a geometric series calculation, where $C=\frac{q-1}{q-q^{-t}}<1$.  If $D'\prec D_{n}$ but $D'\neq D_{n+1}$, then $D'$ does not contain $0$ and hence $f(D')=q^{-nt}$.  Therefore
\begin{equation*}
\begin{split}
V_\Berk(f) 	& = \sum_{n\geq0}\bigg(|f(D_{n+1})-f(D_n)|+ \sum_{\substack{D'\prec D_n \\ D'\neq D_{n+1}}}|f(D')-f(D_n)|\bigg) \\
	& = \sum_{n\geq0}(C(q^{-nt}-q^{-(n+1)t})+(q-1)(q^{-nt}-Cq^{-nt})) \\
	& = \left(C(1-q^{-t})+(q-1)(1-C)\right)\sum_{n\geq0}q^{-nt} \\
	& = \left(C(1-q^{-t})+(q-1)(1-C)\right)\frac{1}{1-q^{-t}} =2C,
\end{split}
\end{equation*}
which is the desired identity.
\end{ex}

The definition of the Berkovich-analytic variation of a Haar-integrable function $f:\Ocal\to\RR$ depends only on the values of integrals of $f$ taken over discs, and therefore $V_\Berk(f)=V_\Berk(g)$ whenever $f=g$ Haar-almost everywhere.  However, functions with finite Berkovich-analytic variation still satisfy a strong regularity condition, as the following result shows.

\begin{thm}\label{FiniteBAVImpliesContinuousThm}
If $f:\Ocal\to\RR$ is Haar-integrable and $V_\Berk(f)<+\infty$, then there exists a unique continuous function $g:\Ocal\to\RR$ such that $f(x)=g(x)$ for Haar-almost all $x\in\Ocal$.
\end{thm}
\begin{proof}
Define a sequence of function $g_n:\Ocal\to\RR$ (for $n\geq0$) by 
\[
g_n(x)=f(D_{1/q^{n}}(x))=q^n\int_{D_{1/q^{n}}(x)} f\,d\mu.
\]
Then define $g:\Ocal\to\RR$ by $g(x)=\lim_{n\to+\infty}g_n(x)$.  To see that this limit exists, note that for $n_1<n_2$ we have
\begin{equation}\label{BerkRegularityCalc1}
\begin{split}
|g_{n_2}(x)-g_{n_1}(x)| & \leq \sum_{n_1\leq n\leq n_2-1}|g_{n+1}(x)-g_{n}(x)| \\
	& \leq \sum_{n\geq n_1}|g_{n+1}(x)-g_{n}(x)| \\
	& \leq \sum_{n\geq n_1}\sum_{\substack{D\in\Dcal_\Ocal \\ \mu(D)=1/q^n}}\sum_{D'\prec D}|f(D')-f(D)|.
\end{split}
\end{equation}
The last expression is the tail of the convergent series (\ref{BAVarDef}) defining Berkovich-analytic variation.  Since the left-hand-side of (\ref{BerkRegularityCalc1}) is majorized by the tail of a convergent series, it follows that the sequence $\{g_n(x)\}$ is Cauchy and hence converges.  Moreover, taking $n_2\to+\infty$ we see that $|g(x)-g_n(x)|$ is also bounded above by the tail of the convergent series (\ref{BAVarDef}), which shows that the convergence $g_n\to g$ is uniform.

If $|x-y|=1/q^n$, then $D_{1/q^n}(x)=D_{1/q^n}(y)$ and so $g_n(x)=g_n(y)$, and therefore
\[
|g(x)-g(y)|\leq |g(x)-g_n(x)|+|g_n(y)-g(y)|\to0
\]
as $n\to+\infty$, proving that $g$ is continuous.

If $D\subseteq\Ocal$ is any disc, then 
\begin{equation*}
\begin{split}
\int_{D}g(x)\,d\mu(x) & = \lim_{n\to+\infty}\int_{D}g_n(x)\,d\mu(x) \\
	& = \lim_{n\to+\infty}q^n\int_{D}\int_{D_{1/q^n}(x)}f(y)\,d\mu(y)\,d\mu(x) \\
	& = \lim_{n\to+\infty}q^n\int_{D}\int_{D_{1/q^n}(y)}f(y)\,d\mu(x)\,d\mu(y) \\
	& = \lim_{n\to+\infty}q^n\int_{D}\mu(D_{1/q^n}(y))f(y)\,d\mu(y) \\
	& = \lim_{n\to+\infty}\int_{D}f(y)\,d\mu(y) \\
	& = \int_{D}f(y)\,d\mu(y).
\end{split}
\end{equation*}
The first equality in the preceding calculation is an interchange of limit and integral which is justified by the dominated convergence theorem, as $g_n\to g$ uniformly and $g$ is continuous and hence bounded.  In the interchange of integrals in the third equality, we note that because $D$ is fixed and $n\to+\infty$, we may assume that $1/q^n\leq\mu(D)$, and in this case
\[
\{(x,y)\mid x\in D \text{ and } y\in D_{1/q^n}(x)\}=\{(x,y)\mid y\in D \text{ and } x\in D_{1/q^n}(y)\}.
\]

Finally, setting $F(x)=f(x)-g(x)$, we have $\int_D F\,d\mu=0$ for all discs $D\subseteq\Ocal$, and therefore $F=0$ Haar-almost everywhere by Lemma~\ref{VanishingIntegralLemma}, completing the proof of the theorem.
\end{proof}

\begin{lem}\label{VanishingIntegralLemma}
If $F:\Ocal\to\RR$ is a Haar-integrable function which satisfies $\int_D F\,d\mu=0$ for all discs $D\subseteq\Ocal$, then $F(x)=0$ for Haar-almost all $x\in\Ocal$.
\end{lem}
\begin{proof}
Any nonempty open subset $U$ of $\Ocal$ is a countable union of discs, and therefore $\int_U F\,d\mu=0$ for all open $U\subseteq\Ocal$.  Taking complements we obtain $\int_K F\,d\mu=0$ for all closed sets $K\subseteq\Ocal$.  

For the sake of obtaining a contradiction, assume that $\{x\in\Ocal\mid F(x)\neq0\}$ has positive Haar measure.  Then at least one of $\{x\in\Ocal\mid F(x)>0\}$ or $\{x\in\Ocal\mid F(x)<0\}$ has positive Haar measure; assume without loss of generality that it is the former.  Since Haar measure is finite on $\Ocal$, it is inner regular, and therefore there exists a closed subset $K\subseteq\{x\in\Ocal\mid F(x)>0\}$ with $\mu(K)>0$. Together, the facts that $\mu(K)>0$ and $F(x)>0$ for all $x\in K$ imply that $\int_K F\,d\mu>0$, a contradiction.
\end{proof}

The converse of Theorem~\ref{FiniteBAVImpliesContinuousThm} is false; that is, not every continuous function has finite Berkovich-analytic variation.  Indeed, in Example~\ref{InfiniteTaibVariation} we constructed a continuous function with infinite Taibleson variation.  By the following result, this continuous function has infinite Berkovich-analytic variation as well.

\begin{prop}\label{TaiblesonLeqBerkProp}
For any continuous $f: \Ocal \to \RR$, we have $V_\Taib(f) \leq V_\Berk(f)$.
\end{prop}
\begin{proof}
The graph-theoretic idea behind this proof is as follows: if $\Pi$ is a Taibleson partition, $D\in\Pi$ is a disc, and $x,y\in D$, then $|f(x)-f(y)|$ is majorized by the variation of $f:\overline{T}_\Ocal\to\RR$ along the interval $I_{x,y}$ in $\overline{T}_\Ocal$ which traverses from $x\in\Ocal$ ``up'' to the smallest disc $D_{|x-y|}(x)=D_{|x-y|}(y)$ containing both $x$ and $y$, and then ``down'' from this disc to $y\in\Ocal$.  Moreover, the intervals $I_{x,y}$ are disjoint for fixed $x,y\in D$ as $D$ ranges over all of the discs in a given Taibleson partition.

To make this precise, let $D_0\subseteq\Ocal$ be a disc, and let $x,y\in D_0$ with $|x-y|=1/q^n$.  To ease the notation define $\alpha_k=f(D_{1/q^k}(x))$ and $\beta_k=f(D_{1/q^k}(y))$.  Since $f$ is continuous, we have $\alpha_k\to f(x)$ and $\beta_k\to f(y)$ as $k\to+\infty$.  Since $D_{1/q^n}(x)=D_{1/q^n}(y)$, we have $\alpha_n=\beta_n$, and thus if $k>n$, a telescoping series calculation gives
\begin{equation*}
\begin{split}
|f(x)-f(y)| & = |f(x)-\alpha_n+\beta_n-f(y)|  \\
 & \leq |f(x)-\alpha_k|+|\beta_k-f(y)|+\sum_{n=i}^{k-1}\left(|\alpha_{i+1}-\alpha_i|+|\beta_{i+1}-\beta_i|\right)
\end{split}
\end{equation*}
Letting $k\to+\infty$ we obtain
\begin{equation*}
\begin{split}
|f(x)-f(y)| & \leq \sum_{i=n}^{\infty}\left(|\alpha_{i+1}-\alpha_i|+|\beta_{i+1}-\beta_i|\right) \\
	& \leq \sum_{D\subseteq D_0}\sum_{D'\prec D}|f(D')-f(D)|.
\end{split}
\end{equation*}
Summing over all discs in a Taibleson partition $\Pi=\{D_1,D_2,\dots,D_M\}$ of $\Ocal$, we have
\begin{equation*}
\begin{split}
\sum_{D_m\in\Pi}\sup_{x,y\in D_m}|f(x)-f(y)| & \leq \sum_{D_m\in\Pi}\sum_{D\subseteq D_m}\sum_{D'\prec D}|f(D')-f(D)| \\
	& \leq \sum_{D\subseteq \Ocal}\sum_{D'\prec D}|f(D')-f(D)|=V_\Berk(f).
\end{split}
\end{equation*}
We obtain the desired inequality $V_\Taib(f) \leq V_\Berk(f)$ by taking the supremum over all Taibleson partitions $\Pi$ of $\Ocal$.
\end{proof}

\begin{ex}
In this example we show that no inequality in the opposite direction of Proposition~\ref{TaiblesonLeqBerkProp} is possible.  As in Example~\ref{AlternatingExample}, for each $k\geq0$, we define  $A_k=D_{1/q^{k+1}}(\pi^k)$, and that the discs $A_k$ are pairwise disjoint.  Define $f:\Ocal\to\RR$ by
\begin{equation}\label{AlternatingExampleFunction2}
f(x)=\sum_{k\geq0}\frac{(-1)^{k}}{k+1}\Xcal_{A_k}(x).
\end{equation}
We showed in Example~\ref{AlternatingExample} that this function is continuous, but not Lipschitz continuous, and that it has finite Taibleson variation.

We can use the graph-theoretic interpretation of Berkovich-analytic variation to give a simple proof that $V_\Berk(f)=+\infty$.  Since $f$ takes the constant value $(-1)^{k}/(k+1)$ on each disc $A_k$, we have $f(A_k)=(-1)^{k}/(k+1)$.  For each $k\geq0$ let $I_k$ be the interval in the tree $T_\Ocal$ formed by the union of three edges: 
\begin{itemize}
\item first the edge from $A_k=D_{1/q^{k+1}}(\pi^k)$ ``up'' to $D_{1/q^k}(0)$;
\item next the edge from $D_{1/q^k}(0)$ ``down'' to $D_{1/q^{k+1}}(0)$
\item and finally the edge from $D_{1/q^{k+1}}(0)$ ``down'' to $A_{k+1}=D_{1/q^{k+2}}(\pi^{k+1})$.
\end{itemize}
The values taken by $f:T_\Ocal\to\RR$ along the interval $I_k$ traverse from $f(A_k)=(-1)^{k}/(k+1)$ to $f(A_{k+1})=(-1)^{k+1}/(k+2)$, and hence the variation of $f$ on $I_k$ is at least $|f(A_k)-f(A_{k+1})|=\frac{1}{k+1}+\frac{1}{k+2}$.  The intervals $I_k$ are disjoint in $T_\Ocal$ and therefore the Berkovich-analytic variation $V_\Berk(f)$ is minorized by the divergent series $\sum_{k\geq0}(\frac{1}{k+1}+\frac{1}{k+2})$, and we conclude that $V_\Berk(f)=+\infty$.
\end{ex}

Finally, we are ready to prove our non-Archimedean analogue of Koksma's inequality using Berkovich-analytic variation.

\begin{thm}\label{Berkovich-Koksma}
If $f: \Ocal \to \RR$ is a continuous, Haar-integrable function and $X$ is a finite subset of $\Ocal$ with discrepancy $\Delta(X)$, then  \begin{equation}
\label{eqn: Berkovich-Koksma}
\left| \frac{1}{N} \sum_{x\in X} f(x) -\int_{\Ocal} f \, d\mu \right| \leq \left(1+\frac{1}{q}\right)V_\Berk(f)\Delta(X). 
\end{equation} 
\end{thm}

\begin{proof}
Let $N=|X|$, and for each disc $D\subseteq \Ocal$, define quantities
\begin{equation*}
\begin{split}
N_D & =|X\cap D| \\
E_D & = \sum_{x\in X\cap D}f(x) - N_Df(D) =\sum_{x\in X\cap D}(f(x)-f(D)).
\end{split}
\end{equation*}
Note that $N_D$ depends on $X$, and $E_D$ depends on both $X$ and $f$, but we suppress these dependencies to ease the notation.  Note also that $N_\Ocal=N$, and the left-hand-side of the desired inequality (\ref{eqn: Berkovich-Koksma}) can be written as $|E_\Ocal|/N$.

The quantity $E_D$ satisfies the identity
\begin{equation}
\label{BerkovichKoksmaRecursion}
E_D = \sum_{D'\prec D}E_{D'}+\sum_{D'\prec D}\left(N_{D'}-\frac{N_D}{q}\right)(f(D')-f(D)),
\end{equation} 
which is elementary to check by simplifying the right-hand-side and using the identities $\sum_{D'\prec D}N_{D'}=N_D$ and $\sum_{D'\prec D}f(D')=qf(D)$.  This can be viewed as a recursion formula for $E_D$ in terms of $E_{D'}$ over the $q$ subdiscs $D'\prec D$.  Since $\mu(D')=\frac{1}{q}\mu(D)$ we have the estimate
\begin{equation*}
\begin{split}
\left|N_{D'}-\frac{N_D}{q}\right| & =\left|N_{D'}-\mu(D')N+\frac{1}{q}(\mu(D)N-N_D)\right| \\
	& \leq N\left(\left|\frac{N_{D'}}{N}-\mu(D')\right|+\frac{1}{q}\left|\frac{N_{D}}{N}-\mu(D)\right|\right) \\
	& \leq N\left(1+\frac{1}{q}\right)\Delta(X),
\end{split}
\end{equation*}
and applying this to (\ref{BerkovichKoksmaRecursion}) we obtain
\begin{equation}
\label{BerkovichKoksmaRecursionEstimate}
|E_D| \leq \sum_{D'\prec D}|E_{D'}|+N\left(1+\frac{1}{q}\right)\Delta(X)\sum_{D'\prec D}|f(D')-f(D)|.
\end{equation} 

Let $M\geq1$ be an arbitrary positive integer.  Iterating the bound (\ref{BerkovichKoksmaRecursionEstimate}) over all discs $D\subseteq \Ocal$ with $1/q^{M-1}\leq\mu(D)\leq1$, we obtain
\begin{equation}
\label{BerkovichKoksmaRecursionEstimateIterated}
\begin{split}
|E_\Ocal| & \leq \sum_{\substack{D\subseteq \Ocal \\ \mu(D)=1/q^{M}}}|E_D| +N\left(1+\frac{1}{q}\right)\Delta(X)\sum_{\substack{D\subseteq \Ocal \\ 1/q^{M-1}\leq\mu(D)\leq1}}\sum_{D'\prec D}|f(D')-f(D)| \\
	& \leq \sum_{\substack{D\subseteq \Ocal \\ \mu(D)=1/q^{M}}}|E_D| +N\left(1+\frac{1}{q}\right)\Delta(X)V_\Berk(f).
\end{split}
\end{equation} 

To complete the proof of the theorem, let $\epsilon>0$ be arbitrary.  Since $f:\Ocal\to\RR$ is continuous on a compact space, it is uniformly continuous, so there exists $M\geq1$ so large that $|f(x)-f(y)|\leq\epsilon$ whenever $|x-y|\leq1/q^M$.  It follows that if $D\subseteq\Ocal$ is a disc of radius $1/q^M$ then 
\[
|E_D|=\left|\sum_{x\in X\cap D}(f(x)-f(D))\right|\leq \epsilon N_D 
\]
and hence
\[
\sum_{\substack{D\subseteq \Ocal \\ \mu(D)=1/q^{M}}}|E_D|\leq \sum_{\substack{D\subseteq \Ocal \\ \mu(D)=1/q^{M}}}\epsilon N_D=\epsilon N
\]
since the discs of radius $1/q^M$ are a partition of $\Ocal$.  Applying this last estimate to (\ref{BerkovichKoksmaRecursionEstimateIterated}) we have
\begin{equation}
\frac{|E_\Ocal|}{N} \leq \epsilon + \left(1+\frac{1}{q}\right)\Delta(X)V_\Berk(f).
\end{equation} 
Since $\epsilon>0$ is arbitrary, we obtain (\ref{eqn: Berkovich-Koksma}), completing the proof.
\end{proof}

\section{Fourier-analytic variation}\label{FourierVarSection}

In this section we derive a Koksma inequality using Fourier analysis on the local ring $\Ocal$.  Rudin \cite{MR0152834} is a standard reference for general Fourier analysis on locally compact abelian groups.  Let $\widehat{\Ocal}$ be the Pontryagin dual group of $\Ocal$; that is, the group of continuous additive characters $\gamma:\Ocal\to\TT$ under pointwise multiplication, where $\TT=\{z\in\CC\mid|z|=1\}$ is the circle group.  Let $\gamma_0\in\widehat{\Ocal}$ denote the trivial character, thus $\gamma_0(x)=1$ for all $x\in \Ocal$.

Given a character $\gamma\in\widehat{\Ocal}$, define the {\bf level} of $\gamma$ to be the smallest nonnegative integer $\ell$ with the property that $\gamma(x)=1$ for all $x\in \pi^\ell \Ocal$.  We denote the level of a character $\gamma\in\widehat{\Ocal}$ by $\ell(\gamma)$.  To see that such an integer always exists, let $\TT^+=\{z\in\TT\mid \mathrm{Re}(z)>0\}$ be the open right unit semicircle. Since $\gamma(0) = 1$ and $\gamma$ is continuous, there exists a neighborhood $\pi^\ell \Ocal$ of zero such that $\gamma(x)\in \TT^+$ for all $x\in\pi^\ell \Ocal$.  Moreover, since $\pi^\ell \Ocal$ is a subgroup of $\Ocal$ we must have that the image $\gamma(\pi^\ell \Ocal)$ is a subgroup of $\TT$. The only subgroup $G$ of $\TT$ entirely contained in $\TT^+$ is the trivial subgroup, because any $z\neq1$ in $\TT$ has the property that some positive power of $z$ has nonpositive real part.  We conclude that $\gamma(\pi^\ell \Ocal) = \{ 1 \}$.

For each $L\geq0$, the set $\widehat{\Ocal}_L=\{\gamma\in\widehat{\Ocal}\mid\ell(\gamma)\leq L\}$ is a subgroup of $\widehat{\Ocal}$.  Note that if $\ell(\gamma)\leq L$, then $\gamma$ factors through the quotient $\Ocal/\pi^L\Ocal$, and this induces an isomorphism between $\widehat{\Ocal}_L$ and the dual group of $\Ocal/\pi^L\Ocal$.  Since finite groups are self-dual, we conclude that $\widehat{\Ocal}_L$ has order $q^L$; in other words, $\widehat{\Ocal}$ contains exactly $q^L$ characters of level at most $L$.  A simple counting argument then shows that for each $\ell\geq1$, $\widehat{\Ocal}$ contains exactly $q^{\ell-1}(q-1)$ characters of level equal to $\ell$.

\begin{lem}\label{CharacterOrthogonalityLemma}
$\,$
\begin{itemize}
\item[{\bf (a)}] For all $L\geq0$ and $x\in \Ocal$, we have
\[
\sum_{\ell(\gamma)\leq L} \gamma(x)= \begin{cases}
q^L & \text{ if } |x|\leq 1/q^L \\
0 & \text{ if } |x|> 1/q^L.
\end{cases}
\]
\item[{\bf (b)}] Let $\gamma\in\widehat{\Ocal}$ be a nontrivial character of level $\ell=\ell(\gamma)\geq1$, and let $c_1,\dots,c_{q^\ell}$ be a complete set of coset representatives for the quotient $\Ocal/\pi^\ell \Ocal$. Then 
\[
\sum_{1\leq i\leq q^\ell} \gamma(c_i)=0.
\]
\end{itemize}
\end{lem}
\begin{proof}
{\bf (a)} If $|x|\leq 1/q^L$ then every character $\gamma$ with $\ell(\gamma)\leq L$ takes the value $1$ at $x$, and there are $q^L$ such characters.  If $|x|> 1/q^L$, then $x$ is nonzero in the finite quotient group $\Ocal/\pi^L\Ocal$, and so there exists a character $\gamma_1\in\widehat{\Ocal}$ of level $\ell(\gamma_1)\leq L$ with $\gamma_1(x)\neq1$.  Since $\widehat{\Ocal}_L=\{\gamma\in\widehat{\Ocal}\mid\ell(\gamma)\leq L\}$ is a subgroup of $\widehat{\Ocal}$, we have
\[
\sum_{\ell(\gamma)\leq L} \gamma(x)=\sum_{\ell(\gamma)\leq L} (\gamma_1\gamma)(x)=\gamma_1(x)\sum_{\ell(\gamma)\leq L} \gamma(x)
\]
which is possible only if $\sum_{\ell(\gamma)\leq L} \gamma(x)=0$, since $\gamma_1(x)\neq1$.

{\bf (b)} Since $\gamma$ has level $\ell$ it factors through the quotient $\Ocal/\pi^\ell \Ocal$ and defines a nontrivial character on that group.  The desired identity then follows from a similar argument as the second case of part {\bf (b)}, with the group $\Ocal/\pi^\ell \Ocal$ in place of $\widehat{\Ocal}_L$.
\end{proof}

Since $\Ocal$ is compact, $\widehat{\Ocal}$ is discrete.  Therefore we may associate to any Haar-integrable function $f:\Ocal\to\CC$ its Fourier series
\begin{equation}\label{FourierSeries}
f(x) \sim \sum_{\gamma\in \widehat{\Ocal}}\hat{f}(\gamma)\gamma(x)
\end{equation}
where the Fourier coefficients of $f$ are defined by
\[
\hat{f}(\gamma)= \int_{\Ocal}f(x)\overline{\gamma(x)}d\mu(x).
\]
The following result states that partial sums of the Fourier series of $f$, ordered with respect to level, converge to $f$ at all points of continuity.  This was proved by Taibleson \cite{MR217522} in characteristic $p$ using more or less the same argument.  We include the proof of the more general statement here.  

\begin{prop}\label{FourierSeriesProp}
If $f:\Ocal\to\CC$ is Haar-integrable and continuous at $x\in \Ocal$, then 
\[
f(x)=\lim_{L\to+\infty}\sum_{\ell(\gamma)\leq L}\hat{f}(\gamma)\gamma(x).
\]
If $f$ is continuous at all $x\in \Ocal$, then the convergence is uniform.
\end{prop}

\begin{proof}
For each $L\geq0$ we define $K_L:\Ocal\to\RR$ by
\begin{equation}\label{DirechletKernelIdentity}
K_L(x) = \sum_{\ell(\gamma)\leq L} \gamma(x) 
= \begin{cases}
q^L & \text{ if } |x|\leq 1/q^L \\
0 & \text{ if } |x|> 1/q^L.
\end{cases}
\end{equation}
The stated identity was proved in Lemma~\ref{CharacterOrthogonalityLemma} {\bf (a)}.  This family of functions may be viewed as an analogue for the local ring $\Ocal$ of the Dirichlet kernel on the circle group $\RR/\ZZ$.  We observe from (\ref{DirechletKernelIdentity}) that $K_L:\Ocal\to\RR$ is nonnegative, $\int_\Ocal K_L\,d\mu=1$, and $K_L(y)=0$ whenever $|y|>1/q^{L}$.

An elementary calculation provides the convolution identity
\begin{equation*}  
(f * K_L) (x) = \int_{\Ocal} f(x-y) K_L(y) \, d\mu(y)= \sum_{\ell(\gamma)\leq L} \hat{f}(\gamma) \gamma(x).
\end{equation*}
Assume that $f$ is continuous at $x$, and let $\epsilon > 0$ be arbitrary.  Then we can find an $L_0$ large enough so that $|f(x-y)-f(x)| < \epsilon$ whenever $|y| \leq 1/q^{L_0}$.  Assuming that $L\geq L_0$, we have
\begin{eqnarray}
\label{eqn: Good-Kernel-Convergence-Zp}
\left|\left(\sum_{\ell(\gamma)\leq L} \hat{f}(\gamma) \gamma(x)\right)-f(x)\right|  &=& \left| \int_{\Ocal} f(x-y) K_L(y) \, d\mu(y) - f(x) \right| \cr
&=& \left| \int_{\Ocal} f(x-y) K_L(y) \, d\mu(y) - \int_{\Ocal} f(x) K_L(y) \, d\mu(y) \right| \cr
&= & \left| \int_{|y|\leq1/p^L} (f(x-y) - f(x)) K_L(y) d\mu(y)  \right| \cr
&\leq& \epsilon\int_{|y|\leq1/p^L} K_L(y) d\mu(y) =\epsilon,
\end{eqnarray}
establishing the desired convergence at $x$.  If $f$ is continuous on $\Ocal$, then it is uniformly continuous since $\Ocal$ is compact. In this case the choice of $L_0$ is independent of $x$ and the convergence $f * K_L\to f$ is uniform.  
\end{proof} 

We are now ready to prove a Fourier-analytic Koksma inequality on $\Ocal$, which may be viewed as a non-Archimedean analogue of a result by Kuipers-Niederreiter; see \cite{MR0419394} p. 161.  Given a Haar-integrable function $f:\Ocal\to\CC$, define its {\bf Fourier-analytic variation} by
\begin{equation*}
V_\Fourier(f)=\sum_{\substack{\gamma\in\widehat{\Ocal} \\ \gamma\neq\gamma_0}}q^{\ell(\gamma)}|\hat{f}(\gamma)|.
\end{equation*}

\begin{thm}\label{thm: Fourier-Koksma-Zp}
Let $f: \Ocal \to \RR$ be a continuous function, and let $X$ be a finite subset of $\Ocal$ with discrepancy $\Delta(X)$.  Then \begin{equation}
\label{eqn: Fourier-Koksma-Zp}
\left| \frac{1}{|X|} \sum_{x\in X} f(x) -\int_{\Ocal} f \, d\mu \right| \leq V_\Fourier(f)\Delta(X). 
\end{equation} 
\end{thm}

\begin{proof} 
Without loss of generality, we may assume that $V_\Fourier(f)<+\infty$, since otherwise (\ref{eqn: Fourier-Koksma-Zp}) holds trivially.  It follows from this assumption and Proposition~\ref{FourierSeriesProp} that the Fourier series (\ref{FourierSeries}) converges absolutely and uniformly to $f(x)$ for all $x\in \Ocal$.  Since $\hat{f}(\gamma_0)=\int_\Ocal fd\mu$, we have 
\begin{equation*}
\begin{split}
\left| \frac{1}{|X|} \sum_{x\in X} f(x)-\int_{\Ocal} f \, d\mu \right| & = \left|  \frac{1}{|X|} \sum_{x\in X} \sum_{\gamma\in \widehat{\Ocal}}\hat{f}(\gamma)\gamma(x) -\int_{\Ocal} f \, d\mu\right| \\
	& =  \left| \sum_{\gamma\neq\gamma_0}\hat{f}(\gamma)\left(\frac{1}{|X|} \sum_{x\in X} \gamma(x) \right)\right| \\
	& \leq \sum_{\gamma\neq\gamma_0}|\hat{f}(\gamma)|\left|\frac{1}{|X|} \sum_{x\in X} \gamma(x)\right|.
\end{split}
\end{equation*}

Let $\gamma\in\widehat{\Ocal}$ be a nontrivial character of level $\ell=\ell(\gamma)\geq1$.  Letting $c_1,\dots,c_{q^\ell}\in \Ocal$ be a set of distinct coset representatives for $\Ocal/\pi^\ell \Ocal$, we have a partition of $\Ocal$ into $q^\ell$ discs $D_1,\dots, D_{q^\ell}$ of radius $1/q^\ell$ centered at $c_1,\dots,c_{q^\ell}$ (respectively), and $\gamma$ is constant on each disc $D_i$.  By Lemma~\ref{CharacterOrthogonalityLemma} {\bf (b)} we have $\sum_{i} \gamma(c_i)=0$, and thus 
\begin{equation*}
\begin{split}
\left| \frac{1}{|X|} \sum_{x\in X} \gamma(x) \right| & =  \left| \sum_{1 \leq i \leq q^\ell} \gamma(c_i) \left( \frac{1}{|X|} \sum_{x\in X}  \Xcal_{D_i} (x)  - q^{-\ell} \right) \right| \\
	& \leq \sum_{1 \leq i \leq q^\ell} \left| \frac{1}{|X|} \sum_{x\in X}  \Xcal_{D_i} (x)  - q^{-\ell} \right| \\
	& \leq q^\ell \Delta(X).
\end{split}
\end{equation*}
Combining the last two displayed estimates we obtain (\ref{eqn: Fourier-Koksma-Zp}).
\end{proof}


\begin{ex} \label{FourierPoweringExample}
Let $c\in \Ocal$, let $t>0$, and consider the function $f:\Ocal\to\RR$ defined by $f(x) = |x-c|^t$.  We show that
\begin{equation}\label{AVVariation}
V_\Fourier(f) = \frac{q^{t}(q^t-1)(q-1)}{(q^{t+1}-1)(q^{t-1}-1)}
\end{equation}
By translation invariance we may assume without loss of generality that $c=0$ and thus $f(x)=|x|^t$.  For each $j\geq0$ define the disc $D_j=\{x\in \Ocal\mid |x|\leq1/q^j\}$ and the circle $S_j=\{x\in \Ocal\mid |x|=1/q^j\}$.  Since $f(x)$ takes the constant value $1/q^{tj}$ on each $S_j$, we have 
\begin{equation*}
\begin{split}
f(x) & = \sum_{j=0}^\infty \frac{1}{q^{tj}} \Xcal_{S_j}(x) \\
	& = \sum_{j=0}^\infty \frac{1}{q^{tj}} \left(\Xcal_{D_j}(x)-\Xcal_{D_{j+1}}(x)\right) \\
	& = \Xcal_{D_0}(x)+\sum_{j=1}^\infty \left(\frac{1}{q^{tj}}-\frac{1}{q^{t(j-1)}}\right)\Xcal_{D_j}(x) \\
	& = \Xcal_{D_0}(x)+\sum_{j=1}^\infty \frac{1-q^t}{q^{tj}}\Xcal_{D_j}(x).
\end{split}
\end{equation*}
For $D=D_{1/q^k}(0)$ and a nontrivial character $\gamma\in\widehat{\Ocal}$, it follows from Lemma~\ref{CharacterOrthogonalityLemma} {\bf (a)} that $\widehat{\Xcal}_D(\gamma)=1/q^k$ when $\ell(\gamma)\leq k$, and $\widehat{\Xcal}_D(\gamma)=0$ otherwise.  Since $\gamma\neq\gamma_0$ we have $\widehat{\Xcal}_{D_0}(\gamma)=0$ and thus 
\begin{equation*}
\begin{split}
\hat{f}(\gamma) & = \sum_{j\geq\ell(\gamma)}\frac{1-q^t}{q^{(t+1)j}} =\frac{C}{q^{(t+1)\ell(\gamma)}},
\end{split}
\end{equation*}
where $C=\frac{q^{t+1}(1-q^t)}{q^{t+1}-1}$.  Since for each $\ell\geq1$ there are exactly $q^{\ell-1}(q-1)$ characters in $\widehat{\Ocal}$ of level equal to $\ell$, we obtain
\begin{equation*}
\begin{split}
V_\Fourier(f) & = \sum_{\gamma\neq\gamma_0}q^{\ell(\gamma)}|\hat{f}(\gamma)| \\
	& = \sum_{\ell\geq1}q^{\ell-1}(q-1)q^\ell\frac{|C|}{q^{(t+1)\ell}} \\
	& = \frac{|C|(q-1)}{q(q^{t-1}-1)},
\end{split}
\end{equation*}
which is (\ref{AVVariation}).
\end{ex}

A standard argument shows that functions with finite Fourier-analytic variation satisfy a strong regularity property, as the following proposition shows; compare with the analogous result for Berkovich-analytic variation, Theorem~\ref{FiniteBAVImpliesContinuousThm}.

\begin{prop}\label{FiniteFAVImpliesContinuousThm}
If $f:\Ocal\to\RR$ is Haar-integrable and $V_\Fourier(f)<+\infty$, then there exists a unique continuous function $g:\Ocal\to\RR$ such that $f(x)=g(x)$ for Haar-almost all $x\in\Ocal$.
\end{prop}
\begin{proof}
It follows from the assumption $V_\Fourier(f)<+\infty$ and Proposition~\ref{FourierSeriesProp} that the Fourier series (\ref{FourierSeries}) converges absolutely and uniformly to a continuous function $g:\Ocal\to\RR$. Since $f$ and $g$ have the same Fourier coefficients, it follows from Parseval's identity on $\Ocal$ that $\int_\Ocal|f-g|^2\,d\mu=0$, which implies that $f=g$ Haar-almost everywhere.
\end{proof}

The following result gives a relationship between the Taibleson and Fourier-analytic variations. In particular, it shows that any real-valued function with rapidly decaying Fourier coefficients must have finite Taibleson variation.

\begin{prop}\label{TaiblesonLeqFourierProp}
For any continuous $f: \Ocal \to \RR$, we have $V_\Taib(f) \leq (2/q)V_\Fourier(f)$.
\end{prop}

\begin{proof}
Without loss of generality, we may assume that $V_\Fourier(f)<+\infty$, since otherwise there is nothing to prove.  It follows from this assumption and Proposition~\ref{FourierSeriesProp} that the Fourier series (\ref{FourierSeries}) converges absolutely and uniformly to $f(x)$ for all $x\in \Ocal$.  Let $X$ be a disc in $\Ocal$ of radius $0<r\leq1$, say $r=1/q^k$ for $k\geq0$. For any pair $x,y\in X$ and any character $\gamma\in\widehat{\Ocal}$ of level $\ell=\ell(\gamma)$, we have
\[
|\gamma(x-y)-1|\leq 2rq^{\ell-1}.
\]
For if $\ell\leq k$, then since $|x-y| \leq r \leq 1/q^\ell$ and $\gamma$ takes the constant value $1$ on $\pi^\ell \Ocal$, we have $\gamma(x-y)=1$.  If on the other hand $\ell>k$, then $q \leq q^{\ell-k}=rq^{\ell}$ and thus $|\gamma(x-y)-1| \leq 2 \leq 2rq^{\ell-1}$. 

It follows that for all $x,y\in X$, we have
\begin{equation}\label{eqn: Variation-FS}
\begin{split}
|f(x) - f(y)| & = \left| \sum_{\gamma \neq \gamma_0} \hat{f}(\gamma) \left( \gamma(x) - \gamma(y) \right) \right| \\
&\leq  \sum_{\gamma \neq \gamma_0} |\hat{f}(\gamma)| \left|\gamma(x-y) -1 \right| \\
&\leq  \sum_{\gamma \neq \gamma_0} |\hat{f}(\gamma)| 2r q^{\ell(\gamma)-1} \\
&= (2/q)r V_\Fourier(f).
\end{split}
\end{equation}
The result follows by applying the upper bound (\ref{eqn: Variation-FS}) to each term in the sum $V_\Pi(f)$ associated to any Taibleson partition, and taking the supremum over all Taibleson partitions.  
\end{proof}

\section{Comparing the Koksma inequalities}\label{ComparingSection}

As a sample application, we consider $f:\Ocal\to\RR$ defined by $f(x)=|x-c|^t$ for $c\in \Ocal$ and $t>0$, and we compare the constants $C(f)$ in the Koksma inequalities of Theorems \ref{thm: Beer-Koksma}, \ref{Berkovich-Koksma}, and \ref{thm: Fourier-Koksma-Zp}, for the Beer variation, Berkovich-analytic variation, and Fourier-analytic variation, respectively.  These were calculated in Examples \ref{BeerPoweringExample}, \ref{BerkPoweringExample}, and \ref{FourierPoweringExample}, respectively. 

\begin{equation*}
\begin{split}
C_\Beer(f) & =2qV_\Beer(f) = 2q(|\alpha-c|^t+|\beta-c|^t)\\
C_\Berk(f) & =\left(1+\frac{1}{q}\right)V_\Berk(f) = \frac{2(q^2-1)}{q(q-q^{-t})} \\
C_\Fourier(f) & =V_\Fourier(f) = \frac{q^{t}(q^t-1)(q-1)}{(q^{t+1}-1)(q^{t-1}-1)} \\
\end{split}
\end{equation*}
Note that because $|\alpha-\beta|=1$, we have $2q\leq C_\Beer(f)\leq 4q$, with both extremes possible depending on the value of $c$.  

We first observe that $C_\Berk(f)< C_\Beer(f)$ is true for all $q$ and $t>0$, and thus the Berkovich-analytic Koksma inequality is always sharper than Beer's result for this family of functions.  

When $t$ is large we have $C_\Fourier(f)\approx q-1$, which is smaller than $C_\Beer(f)$, but not as small as $C_\Berk(f)\approx 2(q^2-1)/q^2$, except for $q=2$.  But as $t\to0$ and hence $f(x)\to1$ (except at $x=c$), we have $C_\Fourier(f)\to0$ but $C_\Berk(f)\to2(q+1)/q$, and thus $C_\Fourier(f)$ is the better constant in this range.

\def\cprime{$'$}



\end{document}